\documentclass[11pt]{amsart}
\usepackage{amsfonts}
\usepackage{amsfonts,latexsym,rawfonts,amsmath,amssymb,amsthm, a4,a4wide}

\usepackage[plainpages=false]{hyperref}

\usepackage{graphicx}

\RequirePackage{color}

\numberwithin{equation}{section}

\newcommand{\beq}{\begin{equation}}
\newcommand{\eeq}{\end{equation}}
\newcommand{\beqs}{\begin{eqnarray*}}
\newcommand{\eeqs}{\end{eqnarray*}}
\newcommand{\beqn}{\begin{eqnarray}}
\newcommand{\eeqn}{\end{eqnarray}}
\newcommand{\beqa}{\begin{array}}
\newcommand{\eeqa}{\end{array}}

\newcommand{\R}{\mathbb R}

\newcommand{\tl}{\tilde\lambda}
\newcommand{\tL}{\tilde\Lambda}

\newcommand{\bb}{\text{\bf{b}}}
\newcommand{\pp}{\text{\bf{p}}}
\newcommand{\e}{\varepsilon}
\newcommand{\p}{\partial}
\newcommand{\eps}{\varepsilon}

{\begin{list}{}%
         {\setlength{\leftmargin}{#1}}%
         \item[]%
}
{\end{list}}

\newtheorem{prop}{Proposition}[section]
\newtheorem{thm}[prop]{Theorem}
\newtheorem{lem}[prop]{Lemma}

\newtheorem{rem}[prop]{Remark}

\renewcommand{\div}{\mbox{div}\,}
\newcommand{\trace}{\mbox{trace}\,}
\newcommand{\dist}{\text{dist}}

\author{Nam Q. Le}
\address{Department of Mathematics, Indiana University, 
Bloomington, IN 47405, USA. }
\email {nqle@indiana.edu}

\thanks{The research of the  author was supported in part by NSF grant DMS-2054686.  }
\allowdisplaybreaks
\title[Twisted Harnack inequality and solvability of singular Abreu equations]{Twisted Harnack inequality and approximation of variational problems with a convexity constraint by singular Abreu equations}

\makeatletter
\@namedef{subjclassname@2020}{%
  \textup{2020} Mathematics Subject Classification}
\makeatother

\begin{document}
\begin{abstract}
 We show in all dimensions that minimizers of variational problems with a convexity constraint,  which arise from the Rochet--Chon\'e model with a quadratic cost in the monopolist's problem in economics, can be approximated in the uniform norm by solutions of singular Abreu equations.
 The difficulty of our Abreu equations consists of having singularities that occur only in a proper subdomain and they cannot be completely removed by any transformations. To solve them, we rely on a new tool which we establish here: a Harnack inequality for singular linearized Monge--Amp\`ere type equations that satisfy certain twisted conditions.
\end{abstract}

\maketitle

\section{Introduction and statement of the main results}
In this paper, we show that 
minimizers of variational problems with a convexity constraint,  which arise from the Rochet--Chon\'e model with a quadratic cost in the monopolist's problem in economics, 
  can be approximated in the uniform norm by solutions of singular Abreu equations. We accomplish this in all dimensions by establishing the solvability of singular Abreu equations in suitable Sobolev spaces, thus resolving an open issue raised in \cite[Remark 1.9]{LePRS}.  A key ingredient in our analysis is a Harnack inequality for solutions of singular linearized Monge--Amp\`ere type equations satisfying certain twisted conditions. 
\subsection{Approximating minimizers of the Rochet--Chon\'e model by singular Abreu equations}
Let $\Omega_0$, $\Omega$ be bounded, open, smooth, and convex domains in $\R^n$ $(n\geq 2)$ where $\Omega$ contains $\overline{\Omega_0}$.
Let $F(x, z, {\textbf p}):\R^n\times \R\times \R^n\rightarrow\R$ be a smooth Lagrangian 
which is convex in each of the variables $z\in\R$ and $ {\textbf p}\in\R^n$. Let $\varphi\in C^{5}(\overline{\Omega})$ be a convex function.

Consider the following variational problem subject to a convexity constraint:
\begin{equation}
\label{minF2}
\inf_{u\in \bar{S}[\varphi,\Omega_0]} \int_{\Omega_0} F(x, u(x), Du(x)) \,dx
\end{equation}
where 
\begin{multline}
\label{barS2}
\bar{S}[\varphi, \Omega_0]=\Big\{ u: \Omega_0\rightarrow \R\mid u \text{ is convex and admits a convex extension to\ } \Omega \\ \text{\ such that }   u=\varphi\text{ on }\Omega\backslash \Omega_0\Big\}.
\end{multline}
This type of problem arises in the analysis of several variational problems including: wrinkling patterns in floating elastic shells in Elasticity \cite{T}, and the Rochet-Chon\'e model \cite{RC} in the monopolist's problem in Economics. A typical Lagrangian for the case of $q$-power cost $(1<q<\infty)$ in the Rochet-Chon\'e model is
\begin{equation}
\label{Fqcost}
F(x, z, {\textbf p}) = (|{\textbf p}|^q/q- x\cdot {\textbf p}+ z)\eta_0(x)
\end{equation}
where $\eta_0$ is the nonnegative relative frequency of agents in the population; see \cite[p. 790]{RC}. 

Because the functions in  $\bar{S}[\varphi, \Omega_0]$ are Lipschitz continuous with Lipschitz constants bounded from above by $\|D\varphi\|_{L^{\infty}(\Omega)}$, one can show, under quite general convexity  and growth assumptions on $F$,  that (\ref{minF2}) has a minimizer in $\bar{S}[\varphi, \Omega_0]$.
Since it is in general difficult to handle the convexity constraint, especially in numerical computations \cite{BCMO, Mir}, one wonders if minimizers of (\ref{minF2}) can be well-approximated in the uniform norm by solutions of some higher order equations whose global well-posedness can be established. 
For the case of Lagrangians $F$ being independent of the gradient variable ${\textbf p}$, this approximation question has been answered in the affirmative by Carlier--Radice \cite{CR} in all dimensions. 
When the Lagrangians $F$ depend on the gradient variable ${\textbf p}$, this has been done \cite{LeCPAM, LePRS, LZ} in two dimensions for a large class of Lagrangians $F$ which include the Rochet--Chon\'e model (\ref{Fqcost}).
The approximating schemes proposed in \cite{CR, LeCPAM, LePRS, LZ} use the second boundary value problem of fourth order equations of Abreu type \cite{Ab}.  The Abreu equation 
\[\sum_{i, j=1}^n U^{ij}\frac{\p^2 [(\det D^2 u)^{-1}]}{\p x_i \p x_j}= f, \quad \text{where } (U^{ij}) = (\det D^2 u) (D^2 u)^{-1},\]
arises in complex geometry in the problem of finding K\"ahler metrics of constant scalar curvature for toric manifolds.   

In our approximating scheme (see \eqref{sAe2} and $J_\e$), Abreu-type equations arise from  the variations of functionals involving
the penalizations $-\e\int_\Omega \log \det D^2u \,dx$ which are convex functionals in the class of $C^2$, strictly convex functions. Heuristically, the logarithm of the Hessian determinant should act as a good barrier for the convexity constraint
in problems like \eqref{minF2}--\eqref{barS2}. This was confirmed numerically in \cite{BCMO} at a discretized level.

In the present work, we address the question of approximating minimizers of (\ref{minF2}) in all dimensions for the gradient-dependent Lagrangians. In particular, our result implies that  
minimizers of the Rochet--Chon\'e model with a quadratic cost can be approximated in the uniform norm by solutions of singular Abreu equations.

Motivated by the Rochet--Chon\'e model (\ref{Fqcost}), and as in \cite{LePRS}, we will focus on Lagrangians of the form
\begin{equation}
\label{F2cost}
F(x, z, {\textbf p}) =  F^0(x, z) + F^1(x, {\textbf p}).
\end{equation}
Here, the function $F^0$ satisfies the following convexity and growth assumptions:
\begin{equation}
\label{F02}
\small
\big( \frac{\p F^0}{\p z}(x, z)- \frac{\p F^0}{\p z}(x, \tilde z)\big)(z-\tilde z)\geq 0;\ \   |F^0(x, z)|+ \big|\frac{\p F^0}{\p z}(x, z)\big|\leq \eta (|z|)~\text{for all } x\in\Omega_0~\text{and } z,\tilde z\in\R
\end{equation}
for some continuous and increasing function $\eta:[0,\infty)\rightarrow [0,\infty)$; the function $F^1$ satisfies the  following convexity and growth assumptions for all ${\textbf p}= (p_1, \cdots, p_n)\in \R^n$: 
\begin{equation}
\label{F1C}
0\leq F^1_{p_i p_j}(x, \pp) \leq D_\ast I_n; |F^1_{p_i x_i} (x, \pp)| \leq D_\ast (|\pp| + 1)\quad\text{for all } x\in\Omega_0, \text{ and for each } i,
\end{equation}
where $I_n$ is the identity $n\times n$ matrix, and $D_\ast$ is a given positive constant. 

In this case, (\ref{minF2}) has a minimizer in $\bar S[\varphi,\Omega_0]$.  The minimizer is unique in the case $F^1_{p_i p_j}(x, \pp)\geq C^{-1} I_n$ for some constant $C>0$.
To approximate minimizers of (\ref{minF2}), we will use
 the following scheme proposed in \cite{LePRS} that is well suited to handling the general case of non-uniform convexity (if any) of the Lagrangian $F$ and the constraint barrier $\varphi$.

Let $\psi\in C^{3}(\overline{\Omega})$ with $\min_{\p\Omega}\psi>0$. Let $\rho$ be a uniformly convex defining function of $\Omega$, that is, 
\begin{equation}
\Omega:=\{x\in \R^n: \rho(x)<0\},~\rho=0 \text{ on } \p\Omega \text{ and }D\rho\neq 0 \text{ on }\p\Omega.
\end{equation} 
Denote by $\chi_{E}$ the characteristic function of a subset $E\subset \R^n$, that is $\chi_E(x) =1$ if $x\in E$ and $\chi_E(x) =0$ if $x\not\in E$.

 For $\e>0$, let \begin{equation}
 \label{mue2}
 \mu_\e(x) =\varphi(x) + \e^{\frac{1}{3n^2}} (e^{\rho(x)}-1), \end{equation}
 and 
 consider the following second boundary value problem for a uniform convex function $u_\e$:
\begin{equation}
\small
\label{sAe2}
  \left\{ 
  \begin{alignedat}{2}\e\sum_{i, j=1}^{n}U_\e^{ij}D_{ij}w_\e~& =f_{\e}\\\ &:=\Big(  \frac{\p F^0}{\p z}(x, u_\e) -\sum_{i=1}^n \frac{\p}{\p x_i} \big(\frac{\p F^1}{\p p_i} (x, Du_\e)\big)\Big)\chi_{\Omega_0} + \frac{u_\e-\mu_\e}{\e}\chi_{\Omega\backslash \Omega_0}~&&\text{in} ~\Omega, \\\
 w_\e~&= (\det D^2 u_\e)^{-1}~&&\text{in}~ \Omega,\\\
u_\e ~&=\varphi~&&\text{on}~\p \Omega,\\\
w_\e ~&= \psi~&&\text{on}~\p \Omega,
\end{alignedat}
\right.
\end{equation}
where \[(U^{ij}_\e)= (\det D^2 u_\e) (D^2 u_\e)^{-1}\] is the cofactor matrix of the Hessian matrix $D^2 u_\e$.

Note that the first two equations of (\ref{sAe2}) consist of a Monge--Amp\`ere equation for $u_\e$:
\begin{equation}
\label{detuwe}
\det D^2 u_\e = w_\e^{-1}\quad\text{in }\Omega, \end{equation}
and a linearized Monge--Amp\`ere equation for $w_\e$:
\begin{equation}
\label{Uwfe}
\sum_{i, j=1}^{n}U_\e^{ij}D_{ij}w_\e=f_\e/\e \quad\text{in }\Omega, \end{equation}
because $U^{ij}_\e D_{ij}$ comes from linearizing the Monge--Amp\`ere operator $\det D^2 u_\e$.  The last two equations of (\ref{sAe2}) prescribe the boundary values for $u_\e$ and $\det D^2 u_\e$ (via $w_\e$). 
Moreover, the first two equations of (\ref{sAe2})
 are critical points, with respect to compactly supported variations,  of  the following convex functional 
\begin{equation*}
J_{\e}(v)=\int_{\Omega_0}  F(x, v(x), Dv(x))\,dx +\frac{1}{2\e}\int_{\Omega\setminus\Omega_0} (v-\mu_\e)^2 \,dx-\e\int_{\Omega}  \log\det D^2 v \, dx.
\end{equation*}
For $u_\e$ being merely convex,  the right-hand side of (\ref{sAe2}) contains the term
\begin{equation}\mathcal{S}u_\e:=-\sum_{i, j=1}^n \frac{\p^2 F^1}{\p p_i \p p_j} (x, Du_\e(x)) D_{ij} u_\e
\end{equation} which is just a nonnegative measure. For this reason, we call (\ref{sAe2}) a {\it singular Abreu equation}. 

The ultimate goal is to solve (\ref{sAe2}) for all small $\e>0$, and then prove that, up to extracting a subsequence of $\e\rightarrow 0$, solutions $u_\e$ tend to  a minimizer $u\in \bar{S}[\varphi,\Omega_0]$ of (\ref{minF2}). This turns out to be highly challenging. One of the mains obstacles comes from the lack of regularity estimates for the 
linearized Monge--Amp\`ere equation (\ref{Uwfe}) with a very singular right-hand side $f_\e$.

As explained in \cite[$\S$1.3]{LeCPAM} and \cite[$\S$1.1]{LZ}, if one directly uses available results for the linearized Monge--Amp\`ere equation, one can only solve (\ref{sAe2}) in two dimensions. Section \ref{HI_stat} will briefly mention this issue. On the other hand, if the singular term $\mathcal{S} u_\e$ is exactly $-\Delta u_\e$ (or a constant multiple thereof; $\mathcal{S}u_\e=-\Delta u_\e$ corresponds to $F^1(x, \pp)=|\pp|^2/2$) and 
$-\Delta u_\e$ appears in the whole domain $\Omega$, rather than just in the subdomain $\Omega_0$, then the recent work
\cite{KLWZ} provides a new technique to solve (\ref{sAe2}) in all dimensions $n\geq 2$. The reason is that when $-\Delta u_\e$ appears in the whole domain $\Omega$, by using positive H\"older continuous functions, one can transform the first two equations of (\ref{sAe2}) into a linearized Monge--Amp\`ere equation with a bounded drift, and the singular term goes away. 
Then one can use H\"older estimates for linearized Monge--Amp\`ere equation with a bounded drift, and the higher order regularity estimates will eventually follow.
This is not the case when  $-\Delta u_\e$ appears only 
in the subdomain $\Omega_0$, since, to the best of the author's knowledge,  it cannot be completely removed by any transformations. In additions, our situation here is more challenging as $F^1$ may depend on the spatial variable $x$, and thus $
\mathcal{S} u_\e$ is not a positive constant multiple of $-\Delta u_\e$ in general.

In this paper, we introduce a new technique to solve (\ref{sAe2}) in all dimensions $n\geq 2$ by overcoming the difficulties mentioned above. The key technical device for obtaining regularity estimates for the 
linearized Monge--Amp\`ere equation (\ref{Uwfe}) with a very singular right-hand side $f_\e$ is a Harnack inequality in Theorem \ref{HI_thm} for singular linearized Monge--Amp\`ere type equations that satisfy certain twisted conditions. This allows us to solve (\ref{sAe2}) in all dimensions $n\geq 2$.

Our key observation here is that, granted positive lower and upper bounds for $\det D^2 u_\e$ have been established, 
\begin{enumerate}
\item $w_\e$ is a supersolution of the linearized Monge--Amp\`ere operator $U_\e^{ij} D_{ij}$ with bounded right-hand side.
\item (Twisted conditions) Up to multiplicative positive H\"older continuous functions,   $w_\e$ is a subsolution of the linearized Monge--Amp\`ere operator $U_\e^{ij} D_{ij}$ with bounded drift and bounded right-hand side. 
\end{enumerate}

Our first main theorem states as follows.
\begin{thm}[Solvability and convergence of singular Abreu equations to minimizers of variational problems with a convexity constraint]
\label{RCthm}
Let $\Omega_0$ and $\Omega$ be bounded, smooth, and convex domains in $\R^n$ ($n\geq 2$) where $\Omega$ is uniformly convex and contains $\overline{\Omega_0}$. 
Let $\varphi \in C^{5}(\overline{\Omega}), \psi\in C^{3}(\overline{\Omega})$ where $\varphi$ is convex, and $\min_{\p\Omega}\psi>0$. Let $F$ be given by (\ref{F2cost}) where
 the smooth function $F^0$ satisfies (\ref{F02}) and the smooth function $F^1$ satisfies (\ref{F1C}). 
If $\e>0$ is small (depending 
only on $n,\Omega,\Omega_0$, $\dist(\Omega_0,\p\Omega)$, $\varphi,\psi,\eta$, and $D_\ast$), then, the following facts hold:
\begin{enumerate}
\item[(i)] The equation (\ref{sAe2}) has  a uniformly convex solution $u_\e \in W^{4,s}(\Omega)$  for all $s\in (n,\infty)$.
\item[(ii)] Let $u_\e\in W^{4, s}(\Omega)$ $(s>n)$ be a solution to (\ref{sAe2}). Then, after extracting a subsequence of $\e\rightarrow 0$, the sequence $\{u_\e\}$ converges uniformly on compact subsets of $\Omega$ to a minimizer $u\in \bar{S}[\varphi,\Omega_0]$  
of (\ref{minF2}).
\end{enumerate}
\end{thm}
We will prove Theorem \ref{RCthm} in Section \ref{RC_sect}.
\begin{rem} Theorem \ref{RCthm}  resolves an open issue raised in \cite[Remark 1.9]{LePRS} for all dimensions $n\geq 3$. Moreover, its conclusions apply to the Rochet--Chon\'e
model with a quadratic cost, that is, (\ref{Fqcost}) with $q=2$.
\end{rem}
We will establish the solvability in Theorem \ref{RCthm}(i) by using a priori estimates and degree theory. For the a priori estimates, 
a key ingredient in our analysis is a Harnack inequality for singular equations of linearized Monge--Amp\`ere type satisfying certain twisted conditions. This is the subject of the next subsection.
\subsection{Singular linearized Monge--Amp\`ere equations with drifts and twisted conditions}
\label{HI_stat}
Let $\lambda,\Lambda,\tl$ and $\tL$ be fixed positive constants where $\lambda\leq \Lambda$ and $\tl\leq \tL$.
Let $\Omega\subset\R^n$ ($n\geq 2$) be an open, convex and bounded domain. Assume that a strictly convex function $u\in C^2(\Omega)$ satisfies the Monge--Amp\`ere equation
\begin{equation}
\label{pinch1}
\lambda\leq \det D^{2} u\leq \Lambda\quad\text{in}~\Omega.
\end{equation}
We will denote the section of $u$ centered at $x\in\Omega$ with height $h>0$  by \[S_u(x, h)=\{y\in \Omega: u(y)< u(x) + Du(x)\cdot (y-x) + h\}.\] 
Throughout, we denote the cofactor matrix of the Hessian matrix $D^2 u= \left(D_{ij} u\right)_{1\leq i, j\leq n}$ by 
$$U = (\det D^2 u)(D^2 u)^{-1}\equiv (U^{ij})_{1\leq i, j\leq n}.$$
Let $A= (a^{ij})_{1\leq i, j\leq n}$  be a symmetric matrix satisfying
\begin{equation}
 \label{AlamU}
\tilde\lambda U\leq A\leq\tilde\Lambda U \quad\text{in}~\Omega.
\end{equation}
We will establish a Harnack inequality for singular linearized Monge--Amp\`ere equations with bounded drifts
\begin{equation}
\label{ALMA}
a^{ij} D_{ij} v + \bb\cdot Dv=f \quad\text{in}~\Omega,\end{equation}
where $f$ is very  singular in general. 

Equations of the type (\ref{ALMA}) include the linearized Monge--Amp\`ere equations when $\tl=\tL=1$, and linear, uniformly elliptic equations when $u(x)= |x|^2/2$ (so that $\lambda=\Lambda=1$).

A Harnack inequality for (\ref{ALMA}) in the case without drifts, $\tl=\tL=1$,  and $f\in L^{r}$ where $r>n/2$ was established in \cite{LN}, and in the case with bounded drifts and $f\in L^{n}$ was established in \cite{LeCCM}. 
Since the validity of a Harnack inequality will imply a H\"older continuity estimate,
the exponent $n/2$ is critical here, as the best regularity for $\Delta v\in L^r$ is $v\in W^{2, r}_{\mathrm{loc}}$ when $r\in (1,\infty)$ and the Sobolev embedding $W^{2, r}\hookrightarrow C^{\alpha}$ only holds when $r>n/2$.

 In applications, as in (\ref{sAe2}), the right-hand side $f$ is a priori only in $L^{1+\hat \e}$ where $\hat \e>0$ is small so $f\not \in L^{n/2}$ when $n\geq 3$. This integrability of $f_\e$ in (\ref{sAe2}) comes from the sharp second-order Sobolev regularity of the Monge--Amp\`ere equation 
 in the works of De Philippis--Figalli--Savin \cite{DFS} and Schmidt \cite{Sc}.
 Thus, in general, there is no Harnack inequality for (\ref{ALMA}) for very singular $f$.
 
However, if we can suitably ``twist'' $v$ to make the right-hand side of (\ref{ALMA}) being bounded, then a Harnack inequality can be established. We consider here twisting by positive multiplicative H\"older continuous functions. The rough idea is as follows. Suppose for each $z\in\Omega$, we can find positive H\"older continuous functions
$\bar \eta^z$, $\underline {\eta}^z$ such that $\bar \eta^z(z)=\underline {\eta}^z(z)=1$, and the {\it ``twisted versions''} of $v$ defined by $\bar v^z:= v \bar \eta^z$ and $\underline {v}^z:= v \underline{ \eta}^z$ satisfy
\[a^{ij} D_{ij} \bar v^z + \bb^z\cdot D\bar v^z\leq \bar f^z, \quad a^{ij} D_{ij} \underline {v}^z + \underline{\bb}^z\cdot D\underline {v}^z\geq \underline{f}^z \]
for suitable functions $\bar \bb^z, \bar f^z$,  and $\underline{\bb}^z, \underline{f}^z$ which are uniformly bounded in $z$. Then, locally, up to positive multiplicative H\"older continuous corrections, $v$ is a solution to a linearized Monge--Amp\`ere equation with bounded 
drifts and bounded right-hand side! Since a Harnack inequality is a local statement, the multiplicative H\"older continuous corrections
should not affect the validity of a Harnack inequality for $v$. This is what we will prove in Theorem \ref{HI_thm} which is of independent interest beyond its application to singular Abreu equations. 

Our choice of terminologies ``twist", ``twisting", and ``twisted" to describe the above multiplication process is motivated by the idea that, within a class of data, the transformed functions in our applications usually change the solution nature of the original functions, from being supersolutions to being subsolutions and vice versa. For example, in \eqref{sup_hr}, $v$ is a supersolution while its transformed (or twisted) versions $\eta^z$ are subsolutions (see \eqref{sub_hzr}) of a class of singular equations of linearized Monge--Amp\`ere type with $L^n$ right-hand side.

Our second main result states as follows.

\begin{thm} [Harnack inequality for singular equations of linearized Monge--Amp\`ere type with twisted conditions]
\label{HI_thm}Assume that (\ref{pinch1}) and (\ref{AlamU}) are satisfied in an open set $\Omega\subset\R^n$. Let $S:=S_u (x_\ast, h)\subset\subset\Omega$ be a section satisfying $S\subset S_u(\bar x, \bar t/2)$ where $S_u(\bar x, \bar t)\subset\subset \Omega$ and 
$B_{R^{-1}}(z)\subset S_u(\bar x, \bar t)\subset B_{R}(z)$ for some $R>1$. 
Suppose that $v\geq 0$ is a $W^{2, n}_{\mathrm{loc}}(\Omega)$ function satisfying the following conditions:
\begin{enumerate}
\item[(i)] It is a supersolution:
\begin{equation}
\label{sup_hr}
a^{ij} D_{ij} v+ \bb\cdot Dv \leq |f| \quad\text{in }S,\end{equation}
where 
$f\in L^n (S)$, and $\bb\in L^\infty(S;\R^n)$. 
\item[(ii)] \emph{(Twisted conditions)} It is a subsolution up to multiplicative H\"older continuous corrections. Precisely, for each $z\in S$, there are functions $\bb^z\in L^\infty(S;\R^n)$, $f^z\in L^n(S)$,  $G^z: S\rightarrow [1, \infty)$ such that
\[G^z(z)=1, \quad \text{and } G^z\in W^{2, n}_{\mathrm{loc}}(S)\cap C^{\gamma}(\overline{S})\quad\text{for some } \gamma\in (0, 1),\] and
 $\eta^z(x)= v(x) G^z(x)$ satisfies
\begin{equation}\label{sub_hzr}a^{ij} D_{ij} \eta^z+ \bb^z\cdot D\eta^z \geq -|f^z| \quad\text{in }S.\end{equation}
\end{enumerate}
Assume that
\begin{equation}
\label{bbzinfty}
\|\bb\|_{L^\infty(S)} + \sup_{z\in S}\|\bb^z\|_{L^\infty(S)} +\sup_{z\in S}\sup_{x\in S }\frac{|G^z(x)-G^z(z)|}{|x-z|^\gamma} \leq K. \end{equation}
Then, there exist positive constants $h_0$ and $C$ depending only on $n,\lambda,\Lambda,\tl, \tL,  R,\gamma,  K$ such that  whenever $h\leq h_0$, we have

\begin{equation}
\label{HI_v}
\sup_{S_{u}(x_\ast, h/8)} v\leq C \Big(\inf_{S_{u}(x_\ast, h/8)} v+ h^{1/2}\|f\|_{L^n(S)}+ h^{1/2}\sup_{z\in S}\|f^z\|_{L^n(S)}\Big).
\end{equation}
In particular, $v$ is locally H\"older continuous with a positive exponent depending on $n$, $\lambda,\Lambda$,$\tl, \tL$, $R,\gamma$, and $K$.
\end{thm}
The proof of Theorem \ref{HI_thm} will be given in Section \ref{HI_sect}.
\begin{rem} When $G^z\equiv 1$, $b^z=b$ and $f^z=f$ for all $z\in S$, Theorem \ref{HI_thm} was proved in \cite[Theorem 1.1]{LeCCM} which extends the fundamental Harnack inequality of the linearized Monge--Amp\`ere equation without drifts in Caffarelli--Guti\'errez \cite{CG}.
\end{rem}
\begin{rem} We have the following remarks on the statement of Theorem \ref{HI_thm}:
\begin{enumerate}
 \item We only consider twisted conditions for the subsolution, and this version of Theorem \ref{HI_thm} suffices for the proof of Theorem \ref{RCthm}. Of course, one can also consider twisted conditions for the supersolution. 
 \item For simplicity, we use the same second order operator $a^{ij} D_{ij}$ in (\ref{sup_hr}) and (\ref{sub_hzr}). The conclusion of Theorem \ref{HI_thm} is unchanged if we use different operators $a^{ij} D_{ij}$ as long as $(a^{ij})$ satisfies (\ref{AlamU}).
 \end{enumerate}
\end{rem}
We briefly explain the proof of Theorem \ref{HI_thm} which follows that of \cite[Theorem 1.1]{LeCCM} where no twisted conditions were involved. First, exactly as in \cite[Theorem 5.3]{LeCCM}, we obtain from the supersolution property (\ref{sup_hr}) the $L^\e$ estimate for $v$. This means that there exists $\e(n,\lambda,\Lambda,\tl, \tL,  R)>0$
such that the distribution function of $v$, $|\{v>t\}\cap S_u(x_\ast, h/4)|$ decays like $t^{-\e}$. Thus, up to constants depending only on $n,\lambda,\Lambda,\tl, \tL,  R$, $v$ is comparable to $v(x_\ast)$ in $S_u(x_\ast, h/4)$ except a set of very small measure. Next, to prove the Harnack inequality, we show, by contradiction, that the maximum of $v$
in $S_u(x_\ast, h/4)$ cannot be much larger than $v(x_\ast)$. Were this not the case, there would exist $\hat x\in S_u(x_\ast, h/4)$ such that $v(\hat x)$ is much larger than $v(x_\ast)$. Then, we apply the $L^\e$ estimate for the supersolution $C_1-C_2 v G^{\hat x}$ (of the operator $a^{ij} D_{ij} + \bb^{\hat x}\cdot D$, due to (\ref{sub_hzr})) to find that $v$ is much larger than $v(x_\ast)$ in a subset of $S_u(x_\ast, h/4)$ of positive measure. This contradicts
the $L^\e$ estimate for $v$.
\vglue 0.4cm

The rest of paper is organized as follows. In Section \ref{RC_sect}, we will establish a priori estimates for solutions to (\ref{sAe2}) and then prove Theorem \ref{RCthm}. In Section \ref{HI_sect}, we will prove Theorem \ref{HI_thm}.

\section{Solvability and convergence of singular Abreu equations}
\label{RC_sect}
In this section, we prove Theorem \ref{RCthm} via a priori estimates and degree theory.
\subsection{A priori higher order derivative estimates}
Fix $s\in (n,\infty)$. Throughout this section, we let  $u_\e \in W^{4,s}(\Omega)$  be a uniformly convex solution to (\ref{sAe2}). Our main result in this section is the global $W^{4, s}(\Omega)$ estimate in Theorem \ref{W4s_thm} for $u_\e$  in terms of the given data $n,\Omega,\Omega_0$, $\varphi,\psi,\eta$, $D_\ast$, and $\e$. As 
the rewriting of (\ref{sAe2}) into (\ref{detuwe}) and (\ref{Uwfe}) indicates, we will obtain regularity estimates for $u_\e$ first, and then for $w_\e$.

From the proof of Theorem 1.4(i) in \cite{LePRS} (see (3.15) there), we have the following uniform estimate.
\begin{prop}[Uniform estimate for $u_\e$]
\label{Linf_prop}
If $\e\leq \e_0$ is sufficiently small, depending 
only on $n,\Omega,\Omega_0,\varphi,\psi,\eta$, $D_\ast$, and $\dist(\Omega_0,\p\Omega)$,  then
\begin{equation}
\label{ue_Linf}
\|u_\e\|_{L^{\infty}(\Omega)}\leq C\end{equation}
where $C$ is independent of $\e$. 
\end{prop}
From now on, we fix $\e\leq \e_0$. 
Unless otherwise stated, constants in this section depend
only on $n,\Omega,\Omega_0$, $\varphi,\psi,\eta$, $\dist(\Omega_0,\p\Omega)$, $D_\ast$, and possibly $\e$. For clarity, the dependence on $\e$ will be explicitly indicated.

Before going further, we establish some simple bounds for $f_\e$ in $\Omega$.
\begin{lem}[Estimates for $f_\e$ in $\Omega$]There is a positive constant $\tilde C$ such that
\begin{equation} 
\label{fOm0}
-\tilde C -D_\ast \Delta u_\e \leq f_\e\leq \tilde C\quad
\text{in }\Omega_0.
\end{equation}
\end{lem}
\begin{proof}
From (\ref{ue_Linf}) and the convexity of $u_\e$, we obtain
\[\|Du_\e\|_{L^{\infty}(\Omega_0)} \leq \frac{2 \|u_\e\|_{L^{\infty}(\Omega)}}{\dist(\Omega_0,\p\Omega)}\leq \frac{2C}{\dist(\Omega_0,\p\Omega)}.\]
Thus, in $\Omega_0$, using the convexity of $u_\e$ together with (\ref{F02}) and (\ref{F1C}), one has 
\begin{eqnarray*}f_\e&=& \frac{\p F^0}{\p z}(x, u_\e(x))  -\sum_{i, j=1}^n \frac{\p^2 F^1}{\p p_i \p p_j} (x, Du_\e(x)) D_{ij} u_\e-\sum_{i=1}^n \frac{\p^2 F^1}{\p x_i \p p_i} (x, Du_\e(x))\\ &\leq&  \frac{\p F^0}{\p z}(x, u_\e(x))  -\sum_{i=1}^n \frac{\p^2 F^1}{\p x_i \p p_i} (x, Du_\e(x))\\&\leq& \eta ( \|u_\e\|_{L^{\infty}(\Omega)})  + nD_\ast (\|Du_\e\|_{L^{\infty}(\Omega_0)}+1 ) \leq \tilde C.\end{eqnarray*}
Similarly, for the lower bound in $\Omega_0$, one has
\[f_\e\geq  \frac{\p F^0}{\p z}(x, u_\e(x))  -D_\ast \Delta u_\e-\sum_{i=1}^n \frac{\p^2 F^1}{\p x_i \p p_i} (x, Du_\e(x))\geq -\tilde C -D_\ast \Delta u_\e.\]
Therefore, we obtain (\ref{fOm0}).
\end{proof}
Combining (\ref{ue_Linf})  with (\ref{fOm0}), we get 
\begin{equation}
\label{fe+}
\|f^{+}_{ \e}\|_{L^{\infty}(\Omega)} \leq C_0(\e)\end{equation}
where
\[f_\e^+=\max\{f_\e, 0\}.\]
 Let us denote
\[(D^2 u_\e)^{-1} = (u_\e^{ij}).\]
The following lemma establishes an upper bound for the Hessian determinant of $u_\e$.
\begin{lem}[Upper bound for $\det D^2 u_\e$] 
\label{detD2ue_up}
There is a positive constant $C_1(\e)$ such that
\[\det D^2 u_\e\leq C_1(\e) \quad\text{in }\Omega.\]
\end{lem}
\begin{proof}
 For this, we use a trick of Chau--Weinkove \cite{CW}. 
Let $$v= \log w_\e-Mu_\e\in W^{2,s}(\Omega)$$
where  $M>0$ is large to be chosen. Then, in $\Omega$, we have
\begin{eqnarray*}u_\e^{ij} D_{ij}v= u_\e^{ij} \left(\frac{D_{ij} w_\e}{w_\e}-\frac{D_i w_\e D_j w_\e}{w_\e^2} -M D_{ij} u_\e\right)&\leq& \frac{u_\e^{ij} D_{ij} w_\e }{w_\e} - nM \\&=& \frac{f_{ \e}/\e}{w_\e \det D^2 u_\e}-n M \\ &\leq&  \frac{\|f^{+}_{ \e}\|_{L^{\infty}(\Omega)}}{\e}-n M  <0,
\end{eqnarray*}
provided $M$ is large. In view of (\ref{fe+}), one can take
\[M=1+ C_0(\e)/\e\geq 1+ \|f^{+}_{ \e}\|_{L^{\infty}(\Omega)}/\e. \]
By the maximum principle, $v$ attains its minimum in $\overline{\Omega}$ on the boundary $\p\Omega$. Thus, in $\Omega$,
\[v \geq \min_{\p\Omega} v\geq \min_{\p\Omega}\log \psi  -M\|u_\e\|_{L^{\infty}(\Omega)} \geq \log \min_{\p\Omega} \psi - MC, \]
where $C$ is the constant in (\ref{ue_Linf}). 
This implies that
\[\log w_\e\geq \log \min_{\p\Omega} \psi - 2MC,\]
so
\[w_\e\geq e^{ \log \min_{\p\Omega} \psi - 2MC} \quad\text{in }\Omega.\]
Therefore, $\det D^2 u_\e= w_\e^{-1}$ is bounded from above by a positive constant $C_1(\e)$.
\end{proof}
Using $u_\e=\varphi$ on $\p\Omega$, and the upper bound for $\det D^2 u_\e$ in Lemma \ref{detD2ue_up}, we can construct suitable barriers to obtain the gradient estimate for $u_\e$:
\begin{equation}
\label{Dueb}
|Du_\e| \leq M_0(\e) \quad \text{in }\Omega.\end{equation}

The following transformation of (\ref{sAe2}), initiated in \cite[Lemma 2.1]{KLWZ}, will be crucially used in our proof of Theorem \ref{RCthm}.
\begin{lem}[Transformed versions of (\ref{sAe2})]
\label{Trlem}
Fix $x_0\in \overline{\Omega}$. Consider the following functions in $\overline{\Omega}$:
\begin{align*}
F^{x_0}_\e (x) &:=\frac{D_\ast|x-Du_\e(x_0)|^2}{2\e},\\
\eta_\e^{x_0}(x) &:= w_\e(x) e^{F^{x_0}_\e (Du_\e(x))},\\
\bb^{x_0}(x)&:=-(\det D^2 u_\e(x))\frac{D_\ast}{\e} (Du_\e(x)-Du_\e(x_0)).
\end{align*}
Then, 
\begin{equation}
\label{etaxeq}
U^{ij}_\e D_{ij} \eta_\e^{x_0}+ \bb^{x_0}(x) \cdot D\eta_\e^{x_0}=\frac{f_\e +D_\ast\Delta u_\e}{\e} e^{F^{x_0}_\e (Du_\e(x))}\quad\text{in }\Omega.
\end{equation}
\end{lem}
\begin{proof}
From \cite[Lemma 2.1]{KLWZ}, we have in $\Omega$
\begin{equation*}U^{ij}_\e D_{ij} \eta_\e^{x_0}-(\det D^2 u_\e)DF_\e^{x_0} (Du_\e) \cdot D\eta_\e^{x_0}=[U^{ij}_\e D_{ij}w_\e + \div (DF_\e^{x_0}(Du_\e)) ]e^{F^{x_0}_\e (Du_\e(x))}.
\end{equation*}
Since
\[DF_\e^{x_0} (Du_\e) = \frac{D_\ast}{\e} (Du_\e-Du_\e(x_0)),\]
and $U^{ij}_\e D_{ij}w_\e =f_\e/\e$, 
the lemma follows.
\end{proof}
Next, we establish a positive lower bound for the Hessian determinant of $u_\e$.
\begin{lem}[Lower bound for $\det D^2 u_\e$]
\label{detD2ue_low}
There is a positive constant $C_2(\e)$ such that
\[\det D^2 u_\e\geq C^{-1}_2(\e)\quad\text{in }\Omega.\]
\end{lem}
\begin{proof} 
We  use the subsolution property of $w_\e$ together with the transformed equation (\ref{etaxeq}) and the Aleksandrov--Bakelman--Pucci (ABP) estimate. 

Fix $x_0\in\Omega$. Let $F_\e^{x_0}$, $\eta_\e^{x_0}$ and  $\bb^{x_0}$ be as in Lemma \ref{Trlem}. Then (\ref{etaxeq}) holds in $\Omega$.

Using (\ref{etaxeq}) together with the first inequality in (\ref{fOm0}) and the uniform estimate for $u_\e$ in (\ref{ue_Linf}), we obtain
\begin{eqnarray}
\label{ueABP}
U^{ij}_\e D_{ij} \eta_\e^{x_0}+ \bb^{x_0}(x) \cdot D\eta_\e^{x_0}&=& \frac{f_\e+ D_\ast  \Delta u_\e}{\e}e^{F^{x_0}_\e (Du_\e(x))}\nonumber\\ &\geq&
\left\{
\begin{alignedat}{2}
-\frac{\tilde C}{\e} e^{F^{x_0}_\e (Du_\e(x))}& \text{ in } \Omega_0\\
\Big[\frac{D_\ast}{\e} \Delta u_\e + \frac{1}{\e^2}\left(u_\e (x)-\mu_\e(x) \right)\Big]e^{F^{x_0}_\e (Du_\e(x))}   & \text{ in } \Omega\backslash \Omega_0
\end{alignedat}
\right.\nonumber\\
&\geq& - D_2(\e). 
\end{eqnarray}

From the ABP estimate  for elliptic, linear equations with drifts (see \cite[inequality (9.14)]{GT}) applied to (\ref{ueABP}),  we find
\begin{multline}
\label{ABPb}
\sup_{x\in \Omega}\eta_\e^{x_0}(x) \leq \sup_{x\in \partial\Omega}\eta_\e^{x_0}(x)\\+ \text{diam}(\Omega)\Bigg\{\exp\Big[\frac{2^{n-2}}{n^n \omega_n} \int_\Omega \Big( 1+ \frac{|\bb^{x_0}|^n}{\det (U_\e^{ij})}\Big)dx \Big]-1\Bigg\}^{1/n}\Big \|\frac{D_2(\e)}{(\det (U_\e^{ij}))^{1/n}}\Big\|_{L^n(\Omega)}
\end{multline}
where $\omega_n$ is the volume of the unit ball $B_1(0)\subset\R^n$.

We now estimate each term appearing in (\ref{ABPb}).  Clearly
\begin{equation}
\label{ABPc}
\sup_{x\in \Omega}\eta_\e^{x_0}(x) \geq \sup_{x\in \Omega} w_\e(x).
\end{equation}
Using (\ref{Dueb}) and $w_\e=\psi$ on $\p\Omega$, we have
\begin{equation}
\label{ABPd}
\sup_{x\in\p \Omega}\eta_\e^{x_0}(x) \leq \sup_{x\in\p \Omega} \psi(x) e^{\frac{2D_\ast \|Du_\e\|^2_{L^{\infty}(\Omega)}}{\e}} \leq M_1(\e).
\end{equation}
Also from (\ref{Dueb}) and the definition of $\bb^{x_0}$, one has
\begin{equation}
\label{bmax}
|\bb^{x_0}(x)| \leq 2\det D^2 u_\e(x) \e^{-1} D_\ast M_0(\e) := D_1(\e)\det D^2 u_\e(x) \quad\text{for all } x\in\overline{\Omega}.
\end{equation}
It follows from (\ref{bmax}) and  Lemma \ref{detD2ue_up} that 
\begin{eqnarray}
\label{ABPe}
 \int_\Omega\Big( 1+ \frac{|\bb^{x_0}(x)|^n}{\det (U_\e^{ij})}\Big)dx&=&  \int_\Omega\Big( 1+ \frac{|\bb^{x_0}(x)|^n}{[\det D^2 u_\e(x)]^{n-1}}\Big)dx \nonumber\\ &\leq& \int_\Omega\Big(1 + D_1^n(\e) \det D^2 u_\e(x)\Big) dx\leq M_2(\e).
\end{eqnarray}
For the last term in (\ref{ABPb}), using \[\det U^{ij}_\e= [\det D^2 u_\e]^{n-1}= w_\e^{-(n-1)},\] we have
\begin{equation}
\label{ABPf}
\Big\|\frac{D_2(\e)}{(\det (U_\e^{ij}))^{1/n}}\Big\|_{L^n(\Omega)} = D_2(\e)\|w_\e^{\frac{n-1}{n}}\|_{L^{n}(\Omega)} \leq D_2(\e)|\Omega|^{\frac{1}{n}} \big(\sup_{x\in \Omega} w_\e(x)\big)^{\frac{n-1}{n}}.\end{equation}
Using (\ref{ABPc}), (\ref{ABPd}), (\ref{ABPe}) and (\ref{ABPf}) in (\ref{ABPb}), we deduce
\[\sup_{x\in \Omega} w_\e(x)\leq M_1(\e) + M_3(\e)  \big(\sup_{x\in \Omega} w_\e(x)\big)^{\frac{n-1}{n}},
\]
for some constant $M_3(\e)>0$.
It follows that $w_\e$ is bounded from above by a constant $C_2(\e)$. Therefore
\[\det D^2 u_\e =w_\e^{-1}\geq C^{-1}_2(\e)\quad\text{in }\Omega.\]
\end{proof}
Now, we prove a H\"older gradient estimate for $u_\e$.
\begin{prop}
[Global $C^{1,\alpha}$ estimates for $u_\e$] 
\label{C1a_prop}
There is $\alpha_0(\e)\in (0, 1)$ such that
\[\|u_\e\|_{C^{1,\alpha_0(\e)}(\Omega)}\leq C_3(\e).\]
\end{prop}
\begin{proof} By Lemmas \ref{detD2ue_up} and \ref{detD2ue_low}, we have
\[C_2(\e)^{-1}\leq \det D^2 u_\e \leq C_1(\e).\] Using these bounds, $u_\e=\varphi$ on $\p\Omega$, and the global $C^{1,\alpha}$ estimates for the Monge--Amp\`ere equation in \cite[Proposition 2.6]{LS}, we obtain the desired global $C^{1,\alpha_0(\e)}$ estimates for $u_\e$.
\end{proof}

The rest of this section is devoted regularity estimates for $w_\e$. 

For the H\"older continuity estimates for $w_\e$ at the boundary, we will use the following one-sided H\"older estimates for solutions to non-uniformly elliptic, linear equations.
\begin{prop}[One-sided pointwise H\"older estimate at the boundary for solutions to non-uniformly elliptic, linear equations with pointwise H\"older continuous drift]\label{sidedH}
Assume that $\Omega\subset\R^n$ is a bounded, uniformly convex domain. Let
$\varphi\in C^{\alpha}(\p\Omega)$ for some $\alpha\in (0, 1)$, and $g\in L^n(\Omega)$. Assume that
the matrix $(a^{ij})$ is  measurable, positive definite and satisfies $\det (a^{ij})\geq \lambda$ in $\Omega$. Let $\bb\in L^{\infty}(\Omega;\mathbb R^n)$.
Let $v\in C(\overline{\Omega})\cap W^{2, n}_{\mathrm{loc}}(\Omega)$ be a function satisfying
$$a^{ij} D_{ij}v + \bb\cdot Dv~\geq -|g| ~\text{in} ~\Omega, \quad
 v= \varphi ~\text{on}~ \p\Omega.
 $$
 Suppose there are constants $\mu,\tau\in (0, 1)$, and $M>0$ such that  at some $x_0\in\p\Omega$, we have
 \begin{equation*}
 |\bb (x)|\leq M|x-x_0|^\mu\quad \text{for all } x\in\Omega\cap B_{\tau}(x_0).\end{equation*}
Then, there exist $\delta, C$ depending only on $\lambda, n, \alpha,\mu, \tau, M$, $\|\bb\|_{L^{\infty}(\Omega)}$, and $\Omega$ such that
\begin{equation*}v(x)-v(x_{0})\leq C|x-x_{0}|^{\frac{\min\{\alpha,\mu\}}{\min\{\alpha,\mu\} +4}}\left(\|\varphi\|_{C^{\alpha}(\p\Omega)} + \|g\|_{L^{n}(\Omega)}\right)~\text{for all}~ x\in \Omega\cap B_{\delta}(x_{0}). \end{equation*}
\end{prop}
\begin{proof} The proof follows from inspecting the proof of Proposition 3.1 in \cite{KLWZ}.
\end{proof}

With the help of Proposition \ref{sidedH} and invoking the key transformations in Lemma \ref{Trlem}, we can now establish the H\"older continuity estimates for $w_\e$ at the boundary.
\begin{prop}[H\"older continuity estimates for $w_\e$ at the boundary] 
\label{bdrw}
There are $\alpha_1(\e)\in (0, 1)$ and $C_4(\e)>0$ such that for all $x_0\in\p\Omega$, we have
\[|w_\e(x) -w_\e(x_0)|\leq C_4(\e)|x-x_0|^{\alpha_1(\e)}\quad\text{for all } x\in \overline{\Omega}\cap B_{C_4^{-1}(\e)}(x_0).\]
\end{prop}
\begin{proof}
Fix $x_0\in \p\Omega$.  Let $F_\e^{x_0}$, $\eta_\e^{x_0}$ and  $\bb^{x_0}$ be as in Lemma \ref{Trlem}. Then (\ref{etaxeq}) holds in $\Omega$.

Recall from (\ref{ueABP}) that
\begin{equation*}
U^{ij}_\e D_{ij} \eta_\e^{x_0} + \bb^{x_0} \cdot D\eta_\e^{x_0}\geq 
 - D_2(\e)\quad\text{in }\Omega. 
\end{equation*}
By Lemma \ref{detD2ue_up} and Proposition \ref{C1a_prop}, the vector field
\[\bb^{x_0}(x):=-(\det D^2 u_\e) \e^{-1}D_\ast (Du_\e(x)-Du_\e(x_0))  \]
satisfies
\[|\bb^{x_0}(x)|\leq \e^{-1}D_\ast C_1(\e) C_3(\e)|x-x_0|^{\alpha_0(\e)} \quad \text{in }\overline{\Omega}.\]
Now, applying Proposition \ref{sidedH}, we have for some $\mu_0(\e)\in (0, 1)$
\begin{equation}
\label{up_eta}
\eta_\e^{x_0}(x)-\eta_\e^{x_0}(x_0) \leq \tilde C(\e)|x-x_0|^{\mu_0(\e)} \quad\text{for all } x\in \overline{\Omega}\cap B_{\tilde C^{-1}(\e)}(x_0).
\end{equation}
Recalling 
$w_\e(x) = \eta_\e^{x_0}(x) e^{-F^{x_0}_\e (Du_\e(x))}$, we have 
\begin{multline*}w_\e(x)-w_\e(x_0) =(\eta_\e^{x_0}(x)-\eta_\e^{x_0}(x_0)) e^{-F^{x_0}_\e (Du_\e(x))} \\+ \eta_\e^{x_0}(x_0) e^{-F^{x_0}_\e (Du_\e(x_0))} (e^{F^{x_0}_\e (Du_\e(x))-F^{x_0}_\e (Du_\e(x_0))}-1).\end{multline*}
Thus, from (\ref{up_eta}) and Proposition \ref{C1a_prop}, 
we obtain one-sided H\"older estimates for $w_\e$ at $x_0$:
\[w_\e(x)-w_\e(x_0)  \leq \hat C(\e) |x-x_0|^{\mu_1(\e)} \quad\text{for all } x\in \overline{\Omega}\cap B_{\hat C^{-1}(\e)}(x_0) \]
where \[\mu_1 (\e)=\min\{\mu_0(\e),\alpha_0(\e)\}.\]
On the other hand, recalling (\ref{fe+}), we also have
\[U^{ij}_\e D_{ij}(-w_\e) =\frac{-f_\e}{\e}\geq \frac{-C_0(\e)}{\e}=-D_3(\e).\]
Applying Proposition \ref{sidedH}, we have the other one-sided H\"older estimates for $w_\e$ at $x_0$:
\[w_\e(x)-w_\e(x_0)  \geq \check C(\e) |x-x_0|^{\mu_2(\e)} \quad\text{for all } x\in \overline{\Omega}\cap B_{\check C^{-1}(\e)}(x_0)\]
for some $\mu_2(\e)\in (0, 1)$.

In conclusion, $w_\e$ is pointwise H\"older continuous at $x_0$ with stated estimate for $C_4(\e)=\max\{\hat C(\e), \check C(\e)\}$.
\end{proof}

Using the twisted Harnack inequality in Theorem \ref{HI_thm}, we will establish interior H\"older estimates for $w_\e$.
\begin{prop}[Interior H\"older estimates for $w_\e$] 
\label{intw}
Let $\Omega'\subset\subset\Omega$. Then,
there is an exponent $\alpha_2(\e)\in (0, 1)$, depending also on $\text{dist}(\Omega',\p\Omega)$, such that $w_\e\in C^{\alpha_2(\e)}(\Omega')$.
\end{prop}
\begin{proof} Let us summarize from the proof of Proposition \ref{bdrw} the following properties of $w_\e$:
\begin{enumerate}
\item {\it $w_\e$ is a supersolution} of the linearized Monge--Amp\`ere operator $U_\e^{ij} D_{ij}$ with bounded right-hand side: \[U^{ij}_\e D_{ij} w_\e \leq D_3(\e)\quad \text{in } \Omega.\]

\item {\it Up to a gauge transformation, $w_\e$ is a subsolution}  of the linearized Monge--Amp\`ere operator, generated by $U_\e^{ij} D_{ij}$, with bounded drift and bounded right-hand side. More precisely, for each $\bar x\in\Omega$, there are functions $\bb^{\bar x}\in L^{\infty}(\Omega;\R^n)$, and  $G^{\bar x} \geq 1$ where $G^{\bar x}(\bar x)=1$ and $G^{\bar x}$ is H\"older continuous on $\overline{\Omega}$ such that $\eta^{\bar x}= w_\e G^{\bar x}$ satisfies
\[U^{ij}_\e D_{ij} \eta^{\bar x} + \bb^{\bar x} (x) \cdot D\eta^{\bar x} \geq -D_2(\e)\quad\text{in }\Omega.\]
The explicit formulas for $G^{\bar x}$ and $\bb^{\bar x}$ are:
\[G^{\bar x}(x)=e^{\frac{D_\ast|Du_\e(x)- Du_\e(\bar x)|^2}{2\e}}; \quad \bb^{\bar x}(x)=-(\det D^2 u_\e) \e^{-1} D_\ast(Du_\e(x)-Du_\e(\bar x)).\]
\end{enumerate}
From Lemma \ref{detD2ue_up} and Proposition \ref{C1a_prop}, we can find a constant $\bar C(\e)$ such that
\[\sup_{\bar x\in \Omega}\|\bb^{\bar x}\| + \sup_{\bar x\in\Omega} \sup_{x\in\Omega} \frac{|G^{\bar x}(x)- G^{\bar x}(\bar x)|}{|x-\bar x|^{\alpha_0(\e)}}\leq \bar C(\e).\]
For any $x_\ast\in\Omega'$, there exist $h(\e), R(\e)>0$, depending also on $\text{dist}(\Omega',\p\Omega)$, such that  $S_{u_\e}(x_\ast, h(\e))\subset\subset\Omega$ satisfies 
\[ B_{R^{-1}(\e)}(x_\ast)\subset S_{u_\e}(x_\ast, h(\e)) \subset B_{R(\e)}(x_\ast).\]
Applying Theorem \ref{HI_thm} with
\[\lambda= C_2^{-1}(\e), \Lambda = C_1(\e), \tilde \lambda =\tilde\Lambda =1, \gamma=\alpha_0(\e), K=\bar C(\e),\]
we find $\alpha_2(\e)\in (0, 1)$, depending also on $\text{dist}(\Omega',\p\Omega)$, such that $w_\e\in C^{\alpha_2(\e)}(\Omega')$.
\end{proof}
Combining the boundary and interior H\"older estimates for $w_\e$, we obtain its global H\"older estimates in the following proposition.
\begin{prop} [Global H\"older estimates for $w_\e$] 
\label{wa_prop}
There is an exponent $\alpha_3(\e)\in (0, 1)$ and $C_5(\e)>0 $ such that \[\|w_\e\|_{C^{\alpha_3(\e)}(\overline{\Omega})}\leq C_5(\e).\]
\end{prop}
\begin{proof} The proposition follows from Propositions \ref{bdrw}  and \ref{intw} using Savin's boundary localization theorem \cite{S1, S2} for the Monge--Amp\`ere equation to connect interior estimates and boundary estimates for linearized Monge--Ampere type equations.
It is similar to the proof of \cite[Theorem 3.2]{KLWZ} so we omit it. 
\end{proof}
Finally, we are ready to establish the global $W^{4, s}$ estimates for $u_\e$.
\begin{thm}[Global $W^{4, s}$ estimates for $u_\e$] 
\label{W4s_thm}
For any $s\in (n, \infty)$, there is a constant $C_\ast(\e, s)$ such that any uniformly convex solution $u_\e \in W^{4,s}(\Omega)$ to (\ref{sAe2}) satisfies
\begin{equation}\label{ueW4s}\|u_\e\|_{W^{4, s}(\Omega)}\leq C_6(\e, s).\end{equation}
\end{thm}
\begin{proof} From \[\det D^2 u_\e= w_\e^{-1}\quad\text{in }\Omega,\quad u_\e=\varphi \quad\text{on }\p\Omega,\] together with Proposition \ref{wa_prop} and the global $C^{2,\alpha}$ estimates for the Monge--Amp\`ere
equation \cite{S2, TW08}, we have
\[\|u_\e\|_{C^{2,\alpha_3(\e)}(\overline{\Omega})} \leq C_6(\e).\]
Thus, the second order operator $U^{ij}_\e$ is uniformly elliptic with $C^{\alpha_3(\e)}(\overline{\Omega})$ coefficients.  Moreover  $\|f_\e\|_{L^{\infty}(\Omega)}\leq \hat C_0(\e)$. Now, from
\[U^{ij}_\e D_{ij} w_\e = f_\e/\e \quad\text{in }\Omega,\quad  w_\e= \psi\quad\text{on }\p\Omega,\]
we can estimate $w_\e$ in $W^{2, s}(\Omega)$. The estimate (\ref{ueW4s}) for $u_\e$ in $W^{4, s}(\Omega)$ follows.
\end{proof}

\subsection{Proof of Theorem \ref{RCthm}} We are now ready to prove Theorem \ref{RCthm}.
\begin{proof}[Proof of Theorem \ref{RCthm}] $(i)$ From the a priori estimates in Theorem \ref{W4s_thm}, one can argue using degree theory as in the proof of \cite[Theorem 1.4(i)]{LePRS} to obtain the existence of a uniformly convex solution $u_\e \in W^{4,s}(\Omega)$ to (\ref{sAe2}) for all $s\in (n,\infty)$. 

$(ii)$ Let $u_\e\in W^{4, s}(\Omega)$ $(s>n)$ be a solution to (\ref{sAe2}). Then, for  $\e\leq \e_0$ small, the uniform estimate (\ref{ue_Linf}) holds. From this and the convexity of $u_\e$, we find that $Du_\e$ is uniformly bounded on each compact subset of $\Omega$. Thus, a subsequence of $u_\e$ converges uniformly on compact subsets of $\Omega$ 
to a convex function $u$ on $\Omega$. As in the proof of \cite[Theorem 1.4(ii)]{LePRS}, we can show that $u$ is a minimizer of (\ref{minF2}).  
\end{proof}

\section{Harnack inequality for singular linearized Monge--Amp\`ere type equations with twisted conditions}
\label{HI_sect}
In this section, we prove Theorem  \ref{HI_thm} following that of \cite[Theorem 1.1]{LeCCM} where no twisted conditions were involved. For reader's convenience, we recall in Subsection \ref{crit_sect} estimates from \cite{LeCCM} for supersolutions of (\ref{ALMA}) with $L^n$ right-hand side: the critical density estimate with $L^n$ drift in Proposition \ref{decay_rem}, the critical density estimate with $L^p$ drift in Lemma \ref{res_lem}, and 
the decay estimate of the distribution function (or $L^\e$ estimate) in Theorem \ref{decay_thm}. They will be used in Subsection \ref{twisted_sect} to prove a Harnack inequality for singular linearized Monge--Amp\`ere equations with small drifts under twisted conditions in Theorem \ref{bsmall_CG}. Theorem  \ref{HI_thm} then follows.

As preliminaries, we recall the volume estimates and  interior $C^{1,\alpha_\ast}$ estimates for the Monge--Amp\`ere equation.  

Assume that the strict convex function $u$ satisfies (\ref{pinch1}). A section $S_u(x, h)$ is called {\it $R$-normalized} if there is $z\in\R^n$ such that \[B_{R^{-1}}(z)\subset S_u(x, h)\subset B_R(z).\]

If $S_u(z, h)\subset\subset\Omega$, then the following {\it volume estimates for section} hold:
 \begin{equation}
 \label{vol_sect} C^{-1}_0(n,\lambda,\Lambda) h^{n/2}\leq |S_u(z, h)|\leq C_0(n,\lambda,\Lambda) h^{n/2};
 \end{equation}
 see \cite[Lemma 4.6]{F}, \cite[Corollary 3.2.5]{G}, and \cite[Theorem 3.42]{LMT}.
 
The strict convexity of $u$ implies that for each $x\in\Omega$, there is $h(x)>0$ such that $S_u(x, h(x))\subset\subset\Omega$. Moreover, this together with (\ref{pinch1}) 
implies the existence of $\alpha_\ast\in (0, 1)$ that depends only on $n,\lambda,\Lambda$ (but not on the $C^2$ character of $u$) such that the following 
{\it interior $C^{1,\alpha_\ast}$ estimates} for $u$ hold: If $S_u(x, h)\subset\subset\Omega$ is $R$-normalized, then 
\begin{equation}
\label{uC1aa}
|Du(y)-Du(z)|\leq C(n,\lambda,\Lambda,  R)|y-z|^{\alpha_\ast}\quad \text{for all } y, z\in S_u(x, h/2);\end{equation}
 see \cite[Theorem 4.20]{F}, \cite[Theorem 5.4.5]{G}, and \cite[Theorem 3.58]{LMT}.

We will prove Theorem \ref{HI_thm} with (\ref{bbzinfty}) being replaced by a slightly relaxed condition:
\begin{equation}
\label{bbzq}
\|\bb\|_{L^q(S)} + \sup_{z\in S}\|\bb^z\|_{L^q(S)} +\sup_{z\in S}\sup_{x\in S }\frac{|G^z(x)-G^z(z)|}{|x-z|^\gamma} \leq K \end{equation}
where
\begin{equation}
\label{paeq}
q>\frac{n(1+\alpha_\ast)}{2\alpha_\ast}:= p_\ast .\end{equation}
This range of $q$ allows us to control the $L^q$ norm of the vector field $\bb$ when rescaling the linearized Monge--Amp\`ere equation (\ref{ALMA});
see (\ref{tbLp}).
In what follows, we fix
\[p= (p_\ast + q)/2.\] 
In the rest of this section, we call a positive constant  {\it universal} if it depends only on $n, q$, $\lambda,\Lambda$, $\tl,\tL$, and $R$. We denote universal constants by $C, C_1, C_2$,  $c_1, c_2$, $\e_3, \e_4$, $M,\delta_\ast, \cdots,$ etc, 
and their values may change from line to
line. We use $C(\cdot, \cdot, \cdot)$ to emphasize the dependence of the constant $C$ on the parameters in the parentheses.

\subsection{Decay estimate of the distribution function for supersolutions}
\label{crit_sect}
This section recalls results from \cite{LeCCM} concerning properties of supersolutions to the linearized Monge--Amp\`ere type equation with $L^n$ drift 
\[a^{ij} D_{ij} v + \bb\cdot Dv\leq f.\]

We first recall the critical density estimate.
\begin{prop}[Critical density estimate for supersolutions with small $L^n$ drift; see  Proposition 5.1 in \cite{LeCCM}]
 \label{decay_rem} Assume that (\ref{pinch1}) and (\ref{AlamU}) are satisfied in $\Omega\subset\R^n$. 
 Suppose that $v\geq 0$ is a $W^{2,n}_{\mathrm{loc}}(\Omega)$ solution of $$a^{ij}D_{ij} v+ \bb \cdot Dv \leq f$$  in an $n$-normalized section $S_u(0, 4t_0)\subset\subset\Omega$. 
 Then, there are small, universal constants $\delta_\ast>0,\e_3>0$ and a large constant $M>1$(all depending only on $n,\lambda,\Lambda,\tl$ and $\tL$) with the following properties. If 
 $$\|\bb\|_{L^n(S_u(0, 4t_0))}+ \|f^{+}\|_{L^n(S_u(0, 4t_0))}\leq \e_3,$$
 and for some nonnegative integer $k$, we have
 \begin{equation*}|\{v>M^{k+1}\}\cap S_u(0, t_0)|> (1-\delta_\ast) |S_u(0, t_0)|, 
 \end{equation*}
then \[v>M^k\quad\text{ in }S_u(0, t_0).\]
\end{prop}
Next, we would like to extend the above critical density estimate to non-normalized sections $S_u(x_0, h)$. For this, we will rescale the linearized Monge--Amp\`ere equation. Due to the degeneracy of (\ref{pinch1}), we will need higher integrability for the vector field $\bb$.

{\bf Rescaling linearized Monge--Amp\`ere equations on a section using John's lemma.}

Assume that (\ref{pinch1}) and (\ref{AlamU}) are satisfied in $\Omega\subset\R^n$. 
We will investigate how the inequalities
\begin{equation}a^{ij} D_{ij} v + \bb\cdot Dv \leq (\geq) f
 \label{eqSh}
\end{equation}
change with respect to normalization, using John's lemma, of a section $S_u(x_0, h)\subset\subset\Omega$ of $u$. 

By subtracting $u(x_0) + Du(x_0)\cdot (x-x_0)+ h$ from $u$, we can assume 
that $u|_{\p S_u(x_0, h)} =0$, and $u$ achieves its minimum value $-h$ at $x_0$. From John's lemma (see, for example \cite[Lemma 3.23]{LMT}), we can 
find an affine transformation \[Tx =A_h x+ b_h\] such that $T^{-1} (S_u(x_0, h))$ is {\it $n$-normalized}; more precisely, we have
\begin{equation} \label{normSh} B_1 (0)\subset T^{-1} (S_u(x_0, h))\subset B_n(0).
\end{equation}
Let us consider the rescaled functions $\tilde u$ of $u$, and $\tilde v$ of $v$ defined by
$$\tilde u(x) =  (\det A_h)^{-2/n} u(Tx),\quad\text{and } \tilde v(x)= v(Tx).$$
Then from (\ref{pinch1}), we have
\begin{equation}
 \label{detD2tilde}
\lambda \leq \det D^2 \tilde u(x)=(\det D^2 u)(Tx) \leq \Lambda~\text{in}~T^{-1} (S_u(x_0, h)),~\tilde u =0 \text{ on }\p T^{-1}(S_u(x_0, h)),
\end{equation}
and
$$B_1(0)\subset \tilde S:=T^{-1}(S_u(x_0, h)) = S_{\tilde u}(y, (\det A_h)^{-2/n} h)\subset B_n(0)$$
where $y$ is the minimum point of $\tilde u$ in $T^{-1} (S_u(x_0, h))$. 

Consider the rescaled coefficient matrix $\tilde A= \left(\tilde a^{ij}\right)_{1\leq i, j\leq n}$ defined by
\begin{equation}
 \label{tildeAA}
 \tilde A= (\det A_h)^{2/n} A_h^{-1} A (A_h^{-1})^{t}.
\end{equation}
Define the rescaled vector field $\tilde \bb$ and function $\tilde f$ by
\begin{equation}\tilde \bb(x) =(\det A_h)^{2/n} A_h^{-1}\bb(Tx),\quad\text{and }\tilde f(x) =(\det A_h)^{2/n} f(Tx).
 \label{bcf}
\end{equation}
Then, the inequalities (\ref{eqSh}) 
become
\begin{eqnarray}
 \label{eqSh1}
 \tilde a^{ij}D_{ij}  \tilde v+ \tilde \bb\cdot D\tilde v  \leq (\geq) \tilde f(x)
\quad\text{in } T^{-1} (S_u(x_0, h)).
\end{eqnarray}

For completeness, we include the derivation of (\ref{eqSh1}). We have
$$D \tilde u = (\det A_h)^{-2/n} A_h^{t} Du;~D^2 \tilde u = (\det A_h)^{-2/n} A_h^{t}D^2 u A_h, $$
and $$D \tilde v = A_h^{t} Dv;~D^2 \tilde v = A_h^{t}D^2 v A_h.$$
The cofactor matrix $\tilde U= (\tilde U^{ij})_{1\leq i, j\leq n}$ of $D^2 \tilde u$ is related to $U=(\det D^2 u) (D^2 u)^{-1}$ and $A_h$ by
\begin{equation}\label{tildeU} \tilde U = (\det D^2 \tilde u)(D^2 \tilde u)^{-1}=
(\det A_h)^{2/n} A_h^{-1} U (A_h^{-1})^{t}. 
\end{equation}
Therefore, from (\ref{tildeAA}), we obtain $$\tilde a^{ij} D_{ij} \tilde v (x)= \trace (\tilde A D^2\tilde v)= (\det A_h)^{2/n} \trace (AD^2 v(Tx))= (\det A_h)^{2/n}  a^{ij} D_{ij} v(Tx).$$
Hence, recalling (\ref{eqSh}), we easily obtain (\ref{eqSh1}) from
\begin{eqnarray*}
 \tilde a^{ij} D_{ij}\tilde v (x) \leq (\geq ) (\det A_h)^{2/n}[f(Tx)- \bb(Tx)\cdot Dv(Tx)]
 = \tilde f(x) -\tilde \bb\cdot D\tilde v.
\end{eqnarray*}

By (\ref{detD2tilde}), (\ref{tildeAA}) and (\ref{tildeU}), $\tilde u$ and $\tilde A$ also satisfy the structural conditions (\ref{pinch1}) and (\ref{AlamU}) on $T^{-1} (S_u(x_0, h))$. 

Using the volume estimates (\ref{vol_sect}), we find from (\ref{normSh}) that 
\begin{equation}
\label{detAh}
[C(n,\lambda,\Lambda)]^{-1} h^{n/2}\leq \det A_h
\leq C(n,\lambda,\Lambda) h^{n/2}.
\end{equation}

{\it Assume now $S_u(x_0, h)\subset S_u(\bar x, \bar t/2)$ where $S_u(\bar x, \bar t)\subset\subset\Omega$ is an $R$-normalized section. }

From the volume estimates of sections, we can find $c_{\ast}(n,\lambda,\Lambda, R)$ such that
\begin{equation}
\label{cadefn}
h\leq c_\ast
\text{ whenever }S_u(x_0, h)\subset S_u(\bar x, \bar t/2), \quad\text{where }S_u(\bar x, \bar t)\subset\subset\Omega \text{ is an R-normalized}.
\end{equation}

In \cite{LeCCM}, a similar choice for $c_\ast$ was discussed before the statement of Lemma 5.2. Here it also serves the role of $h_0$ in Theorem 6.1 in \cite{LeCCM}.

Note that the interior $C^{1,\alpha_\ast}$ estimate for $u$  in (\ref{uC1aa}) implies that
 $$S_u(x_0, h)\supset B_{c_1h^{\frac{1}{1+\alpha_\ast}}}(x_0)$$ for some universal constant $c_1= c_1(n,\lambda,\Lambda, R)$. This combined with (\ref{normSh}) gives
 \begin{equation}
   \label{Anorm}
  \|A_h^{-1}\|\leq C(n,\lambda,\Lambda, R)h^{-\frac{1}{1+\alpha_\ast}}.
 \end{equation}
By (\ref{bcf}), (\ref{detAh}) and (\ref{Anorm}), we can estimate 
\begin{eqnarray} \label{tbLp}\|\tilde \bb\|_{L^p(\tilde S)}\leq C h \| A_h^{-1}\bb(Tx)\|_{L^p(\tilde S)}&\leq&
Ch^{1-\frac{n}{2p}}\|A_h^{-1}\|\|\bb\|_{L^p(S_u(x_0, h))}\nonumber\\ &\leq& C(n,\lambda,\Lambda, R) h^{\frac{\alpha_\ast}{1+\alpha_\ast}-\frac{n}{2p}}\|\bb\|_{L^p(S_u(x_0, h))}.
\end{eqnarray}
 Then, by the H\"older inequality, we have
\begin{equation} \|\tilde \bb\|_{L^n(\tilde S)}\leq |\tilde S|^{\frac{1}{n}-\frac{1}{p}}\|\tilde \bb\|_{L^p(\tilde S)}\leq 
C(n,\lambda,\Lambda, R, p) h^{\frac{\alpha_\ast}{1+\alpha_\ast}-\frac{n}{2p}}\|\bb\|_{L^p(S_u(x_0, h))}.
\label{tbLn}
\end{equation}
Recalling (\ref{detAh}), we can estimate
\begin{equation} \|\tilde f\|_{L^n(\tilde S)}=(\det A_h)^{2/n} h^{-1/2} \|f\|_{L^n(S_u(x_0, h))}\leq C(n,\lambda,\Lambda)h^{1/2} \|f\|_{L^n(S_u(x_0, h))}.
 \label{tfLn}
\end{equation}
Due to \[p>p_\ast = \frac{n(1+\alpha_\ast)}{ 2\alpha_\ast},\]
we deduce that  in our rescaling process, the $L^n$ norms of $\bb$ and $f^{+}$ are small if $h$ is small, or $\|\bb\|_{L^p(S_u(x_0, h))} + \|f^{+}\|_{L^n(S_u(x_0, h))}$ is small.

Thus, rescaling and Proposition \ref{decay_rem} give the following result.
\begin{lem} [Critical density estimate  for supersolution with small $L^p$ drift; see Lemma 5.2 in \cite{LeCCM}]
\label{res_lem}
 Assume that (\ref{pinch1}) and (\ref{AlamU}) are satisfied in $\Omega\subset\R^n$. Let $p>p_\ast$ where $p_\ast$ is as in (\ref{paeq}).
There is a small number $\e_4$ depending only on $p,n,R, \lambda,\Lambda,\tl$ and $\tL$ with the following property.
 Suppose that $v\geq 0$ is a $W^{2, n}_{\mathrm{loc}}(\Omega)$ solution of  $$a^{ij}D_{ij} v+ \bb\cdot Dv \leq f$$  in a section 
 $S_u(x_0, h)\subset S_u(\bar x, \bar t/2)$ where $S_u(\bar x, \bar t)\subset\subset\Omega$ is an $R$-normalized section,
 and 
 \begin{equation*}\|\bb\|_{L^p(S_u(x_0, h))} + \|f^{+}\|_{L^n(S_u(x_0, h))}\leq \e_4.
 \end{equation*}
Let $M$ and $\delta_\ast$ be as in Proposition \ref{decay_rem}. If
 for some nonnegative integer $k$, we have
 \begin{equation*}|\{v>M^{k+1}\}\cap S_u(x_0, h/4)|> (1-\delta_\ast) |S_u(x_0, h/4)|,
 \end{equation*}
then \[v>M^k\quad \text{in }S_u(x_0, h/4).\]
\end{lem}

 By \cite[Lemma  2.14]{LeCCM}, there exists a universal constant $\hat K(n,\lambda,\Lambda)>1$ with the following property:
 \begin{center}
If $S_u(x, t)\subset S_u(y, h)$ where $S_u(y, \hat K h)\subset\subset \Omega$ then $S_u(x, 4t)\subset S_u(y, \hat K h).$ 
\end{center}
We have the following result on the decay estimate of the distribution function (or $L^\e$ estimate)  of supersolutions.
\begin{thm}[Decay estimate of the distribution function of supersolutions with small $L^p$ drift; see Theorem 5.3 in \cite{LeCCM}] \label{decay_thm} 
 Assume that (\ref{pinch1}) and (\ref{AlamU}) are satisfied in $\Omega\subset\R^n$. Let $p>p_\ast$ where $p_\ast$ is as in (\ref{paeq}). 
 Let $\e_4$ and $c_\ast$ be as in Lemma \ref{res_lem}.
 Suppose that $v\geq 0$ is a $W^{2,n}_{\mathrm{loc}}(\Omega)$ solution of $$a^{ij}D_{ij} v + \bb \cdot Dv \leq f$$  in a section $S_4=S_u(0, 4 t_0)\subset\subset\Omega$ with
 $S:= S_u(0, \hat K t_0))\subset S_u(\bar x, \bar t/2)$  where $S_u(\bar x, \bar t)\subset\subset\Omega$ is an $R$-normalized section, and 
 \begin{equation*}\|\bb\|_{L^p(S)} + \|f^{+}\|_{L^n(S)}\leq \e_4.
 \end{equation*}
Suppose that
 $$\inf_{S_u(0, t_0)} v\leq 1.$$
Then there are universal constants $C_1(n,\lambda,\Lambda,\tl,\tL, R)>1$  and $\e(n,\lambda,\Lambda,\tl,\tL, R)\in (0, 1)$ such that
\[ |\{ v > t \} \cap S_u(0, t_0) | \leq C_1 t^{-\eps} |S_u(0, t_0)|~\text{for all~} t>0.\]
\end{thm}

\subsection{Harnack inequality for singular linearized Monge--Amp\`ere equation with small drifts under twisted conditions}
\label{twisted_sect}
Let $M>1$ be the constant in Proposition \ref{decay_rem}. Our main result in this section is the following Harnack inequality  under twisted conditions.

\begin{thm} [Harnack inequality for singular linearized Monge--Amp\`ere type equation with small drifts  under twisted conditions]
\label{bsmall_CG}Assume that (\ref{pinch1}) and (\ref{AlamU}) are satisfied in $\Omega$. 
Suppose that $v\geq 0$ is a $W^{2, n}_{\mathrm{loc}}(\Omega)$ function satisfying the following conditions  in 
a section \[S:=S_u (x_\ast, h)\subset S_u(\bar x, \bar t/2)\] where $S_u(\bar x, \bar t)\subset\subset\Omega$ is an $R$-normalized section:
\begin{enumerate}
\item[(i)] It is a supersolution:
\begin{equation}
\label{sup_h}
a^{ij} D_{ij} v+ \bb\cdot Dv \leq |f|\quad\text{in } S\end{equation}
where 
$f\in L^n (S)$, and $\bb\in L^p(S;\R^n)$ 
with $p>p_\ast$ where $p_\ast$ is as in (\ref{paeq}). 
\item[(ii)] It is a subsolution up to positive multiplicative continuous corrections. Precisely, for each $z\in S$, there are functions $\bb^z\in L^p(S;\R^n)$, $f^z\in L^n(S)$,  $G^z: S\rightarrow [1, \infty)$ such that
\begin{equation}\label{GzM}G^z(z)=1, \quad G^z\in W^{2, n}_{\mathrm{loc}}(S), \quad  \sup_{x\in S_u(x_\ast, h)}|G^z(x) -G^z(z)|\leq \frac{1}{3M+1},\end{equation} and
 \[\eta^z(x)= v(x) G^z(x)\] satisfies
\begin{equation}\label{sub_hz}a^{ij} D_{ij} \eta^z+ \bb^z\cdot D\eta^z \geq -|f^z| \quad\text{in } S.\end{equation}
\end{enumerate}
There exists a universal constant $\e_5(n,p,\lambda,\Lambda,\tl,\tL, R)>0$ with the following property.  
If 
$$\|\bb\|_{L^p(S)}\leq \e_5/2,\quad\sup_{z\in S}\|\bb^z\|_{L^p(S)}\leq \e_5/2$$
then
\begin{equation}
 \label{HI3}
\sup_{S_{u}(x_\ast, h/8)} v\leq C(n, \lambda, \Lambda,\tl,\tL, R) \Big(\inf_{S_{u}(x_\ast, h/8)} v+ h^{1/2}\|f\|_{L^n(S)}+ h^{1/2}\sup_{z\in S}\|f^z\|_{L^n(S)}\Big).
\end{equation}
\end{thm}

\begin{rem} In (\ref{GzM}), we only need the small oscillation of $G^z$ around $z$, and no continuity properties are required.
\end{rem}

\begin{proof}[Proof of Theorem \ref{bsmall_CG}] 
Let $\delta_\ast\in (0, 1)$  be the constant in Proposition \ref{decay_rem} and $\e\in (0, 1)$ be the constant in Theorem \ref{decay_thm}. Let $c_\ast$ be as (\ref{cadefn}), and
let $\e_4$ be as in Lemma \ref{res_lem}. We choose $\e_5$ so that $$Cc_\ast^{\frac{\alpha_\ast}{1+\alpha_\ast}-\frac{n}{2p}}\e_5\leq \frac{\e_4}{16 M}$$
where $C$ is the universal constant appearing in (\ref{tbLp})- (\ref{tfLn}).
We rescale (\ref{sup_h}), (\ref{sub_hz}), the domain, and functions as in Section \ref{crit_sect}. In particular, for $x\in T^{-1}(S_u(x_\ast, h))$, we have
\[\tilde u(x) =(\det A_h)^{-2/n} u(Tx), \quad \tilde v(x) = v(Tx),\quad\text{and } \tilde G^{z}(x)= G^z(Tx).\]
The corresponding matrix $\tilde A=(\tilde a^{ij})$ is given by (\ref{tildeAA}), and $\tilde \bb$, $\tilde \bb^z$, $\tilde f$ and $\tilde f^z$ are given by (\ref{bcf}).

By (\ref{tbLp}), the functions $\tilde \bb$, $\tilde \bb^z$
satisfy on the {\it $n$-normalized} section \[\tilde S= T^{-1}(S_u(x_\ast, h))= S_{\tilde u}(y, 4t_0)\] (see (\ref{normSh})) the bounds
\begin{equation}\|\tilde \bb\|_{L^p(\tilde S)} \leq Cc_\ast^{\frac{\alpha_\ast}{1+\alpha_\ast}-\frac{n}{2p}}\e_5\leq \frac{\e_4}{16 M}, \quad \sup_{z\in S}\|\tilde \bb^z\|_{L^p(\tilde S)} \leq Cc_\ast^{\frac{\alpha_\ast}{1+\alpha_\ast}-\frac{n}{2p}}\e_5\leq \frac{\e_4}{16 M}.
 \label{bcsmall}
\end{equation}
From (\ref{sup_h}), we have the supersolution property
\[\tilde a^{ij} D_{ij}\tilde v + \tilde \bb\cdot D\tilde v\leq| \tilde f|.\]
For each $z\in S=T(\tilde S)$, we have from (\ref{sub_hz}) the following subsolution property
\[\tilde a^{ij} D_{ij}\tilde \eta^{z} + \tilde \bb^z\cdot D\tilde \eta^{z}\geq -| \tilde f^{z}|.\]
We also know that
\begin{equation}\label{tGz}\tilde G^z (T^{-1} z) =1,\quad \quad  \sup_{x\in S_u(x_\ast, h)}|\tilde G^z(T^{-1}x) -\tilde G^z(T^{-1}z)|\leq \frac{1}{3M+1}.\end{equation}
We need to show that
\begin{equation}\sup_{S_{\tilde u}(y, t_0/2)} \tilde v\leq C(n, \lambda, \Lambda,\tl,\tL) \left(\inf_{S_{\tilde u}(y, t_0/2)} \tilde v+ \|\tilde f\|_{L^n(\tilde S)}+ \sup_{z\in S}\|\tilde f^z\|_{L^n(\tilde S)}\right).
 \label{HI4}
\end{equation}
Since $S_{\tilde u}(y, 4t_0)$ is $n$-normalized, and (\ref{detD2tilde}) holds, we can use the volume estimate (\ref{vol_sect}) to obtain
\[C^{-1}(n,\lambda,\Lambda) \leq t_0\leq C(n,\lambda,\Lambda).\]
Without loss of generality, we can assume that $t_0=1$ and $y=0$. 
By changing coordinates and subtracting an affine function from $\tilde u$,
we can assume that \[\tilde u\geq 0, \quad \tilde u(0)=0, \quad D\tilde u(0)=0.\]
We divide the proof into several steps.\\
{\it Step 1.} We show that,  if \begin{equation}\label{small1}\inf_{S_{\tilde u}(0, 1/2)} \tilde v \leq 1, \quad \text{ and}~ \|\tilde f\|_{L^n(\tilde S)}+ \sup_{z\in S}\|\tilde f^z\|_{L^n(\tilde S)}\leq \frac{\e_4}{16 M},\end{equation} then for some universal constant $C$, we have \begin{equation}
\label{tvC}
\sup_{S_{\tilde u}(0, 1/2)} \tilde v \leq C.\end{equation}

Our proof of (\ref{tvC}) follows the lines of argument in Imbert-Silvestre \cite{IS} in the case $\tilde u(x)=|x|^2/2$.
Let $\beta >0$ be a universal constant to be determined later and let \[h_t(x) = t(1-\tilde u(x))^{-\beta}\quad\text{for } x\in S_{\tilde u}(0, 1).\] 
We consider the minimum value of $t>0$ such that
$h_t \geq \tilde v$ in $S_{\tilde u}(0, 1)$. It suffices to show that $t$ is bounded from above by a universal constant $C$, because we have then
$$\sup_{S_{\tilde u}(0, 1/2)} \tilde v\leq C \sup_{S_{\tilde u}(0, 1/2)} (1- \tilde u)^{-\beta}\leq 2^\beta C.$$
If $t \le 1$, then we are
done. Hence, it remains to prove (\ref{tvC}) for the case $t \geq 1$. 

Since $t$ is chosen to be the minimum value such that $h_t \geq
\tilde v$, there is $x_0 \in S_{\tilde u}(0, 1)$ such that $h_t(x_0) =
\tilde v(x_0)$. Let \[r = (1-\tilde u(x_0))/2,\quad\text{and }H_0 := h_t(x_0) = t(2r)^{-\beta}
\geq 1.\] By the inclusion and exclusion property of sections (see \cite[Theorem 3.3.10]{G}, and \cite[Theorem 3.57]{LMT}),
there is a small constant $\hat c$ and large constant $p_1$ depending on $n,\lambda,\Lambda$ such that 
\begin{equation}
 \label{sectx0}
S_{\tilde u}(x_0, \hat K \hat c r^{p_1})\subset S_{\tilde u}(0, 1)
\end{equation}
where $\hat K$ is defined right before the statement of Theorem
\ref{decay_thm}.

We bound $t$ from above by estimating the measure of the set $\{\tilde v \geq H_0/2\} \cap
S_{\tilde u}(x_0, \hat c r^{p_1})$ from above and below. 

The estimate from above can be done using Theorem
\ref{decay_thm}. First, recalling $\tilde S= S_{\tilde u}(0, 4)$, we find from (\ref{bcsmall}), (\ref{small1}) and (\ref{sectx0})
that
$$\|\tilde \bb\|_{L^p(S_{\tilde u}(x_0, \hat K \hat c r^{p_1}))} + \|\tilde f\|_{L^n(S_{\tilde u}(x_0, \hat K \hat c r^{p_1}))}\leq \e_4.$$
Then, Theorem \ref{decay_thm} gives
\begin{equation} \label{up_H}  
|\{\tilde v>H_0/2\} \cap S_{\tilde u}(x_0, \hat c r^{p_1})|\leq
CH_0^{-\eps}|S_{\tilde u}(x_0, \hat c r^{p_1})|\leq C H_0^{-\eps}|S_{\tilde u}(0, 1)| \leq C t^{-\eps} (2r)^{\beta \eps}.
\end{equation}

Now, we estimate the measure of $\{\tilde v \geq H_0/2\} \cap
S_{\tilde u}(x_0, \hat c r^{p_1})$ from below. To do this, we modify the proof of Theorem 6.1 on pp. 35--37 in \cite{LeCCM}.
The only part that is different starts right after inequality (6.7). Here, $C_1-C_2 \tilde v$ is not a supersolution of the operator $\tilde a^{ij} D_{ij} + \tilde \bb\cdot D$. However, \[C_1-C_2 \tilde v\tilde G^{Tx_0}=C_1-C_2 \tilde \eta^{Tx_0}\] is a supersolution of
the operator $\tilde a^{ij}D_{ij}  + \tilde \bb^{Tx_0} \cdot D$; see (\ref{weq}).

We will appropriately choose $C_1$ and $C_2$ (see (\ref{wC1C2})), and then applying Lemma \ref{res_lem} to $C_1-C_2 \tilde v\tilde G^{Tx_0}$ on a small but definite fraction of the section $S_{\tilde u}(x_0, \hat c r^{p_1})$. For this, we introduce several new constants. 
Denote
\[\delta:= \frac{1}{3M+1}<\frac{1}{3M}
\quad\text{so }1+\delta <4/3.\]
Let $\rho$ be a small universal
constant, and let $\beta$ be a large universal constant such that
\begin{equation}\label{beta_choice}
 M \left( (1+\delta)(1-\rho)^{-\beta} - 1 \right)  =
 \frac{1}{3};~
\beta  \geq \frac{np_1}{2\e}.
\end{equation}
As in \cite[p. 36]{LeCCM}, we deduce from the interior $C^{1,\alpha_\ast}$ estimate (\ref{uC1aa}) and the gradient estimate for $u$ that
 \[1- \tilde {u}(x)\geq 2r-2\rho r\quad  \text{in the section }S_{\tilde u}(x_0, c_1 r^{p_1}) \] if $c_1$
is universally small. 

Note that \[\tilde v(x_0)= H_0\geq 1.\]
The maximum of $\tilde v$ in the section $S_{\tilde u}(x_0, c_1 r^{p_1})$ is at most the maximum
of $h_t$ in $S_{\tilde u}(x_0, c_1 r^{p_1})$ which is not greater than $t(2r-2\rho r)^{-\beta} =
(1-\rho)^{-\beta} H_0$. Thus
\[\tilde v\leq (1-\rho^{-\beta}) H_0\quad\text{in } S_{\tilde u}(x_0, c_1 r^{p_1}).\]
By (\ref{tGz}), we have \[G_m:= \max_{S_{\tilde u}(x_0, c_1 r^{p_1})} \tilde G^{Tx_0}\leq 1+\delta.\]
Define the 
following functions for $x\in S_{\tilde u}(x_0, c_1 r^{p_1}) $

\begin{equation} \label{wC1C2} w(x) = \frac{(1-\rho)^{-\beta} H_0 G_m
  -\tilde v(x)\tilde G^{Tx_0}}{\left((1+\delta)(1-\rho)^{-\beta} - 1
  \right) H_0},~\text{and }\bar f =\frac{|\tilde f^{Tx_0}|}{\left((1+\delta)(1-\rho)^{-\beta} - 1
  \right) H_0}. 
  \end{equation}
   Note that $w(x_0) \leq 1$, and $w$ is a non-negative solution of 
\begin{equation} \label{weq}\tilde a^{ij}D_{ij} w + \tilde \bb^{Tx_0} \cdot Dw \leq  \bar f~\text{in}~S_u(x_0, c_1 r^{p_1}). 
\end{equation}
Observe that, by (\ref{beta_choice}), (\ref{bcsmall}) and the assumption on $\tilde f^{Tx_0}$ in (\ref{small1}),
\begin{equation*}
 \|\bar f\|_{L^n(\tilde S)}\leq  \frac{1}{\left((1+\delta)(1-\rho)^{-\beta} - 1
  \right) } \|\tilde f^{Tx_0}\|_{L^n(\tilde S)}
  =  3M \|\tilde f^{Tx_0}\|_{L^n(\tilde S)} \leq \e_4/2.
\end{equation*}
Therefore,
\begin{equation}\|\tilde \bb^{Tx_0}\|_{L^p(S_{\tilde u}(x_0, c_1 r^{p_1}))}+ \|\bar f\|_{L^n(S_{\tilde u}(x_0, c_1 r^{p_1}))}\leq \e_4.
 \label{bcbarf}
\end{equation}
From (\ref{weq}) and (\ref{bcbarf}), we can use Lemma \ref{res_lem} to obtain the estimate
\begin{equation} |\{w \leq M\} \cap S_{\tilde u}(x_0, 1/4 c_1 r^{p_1})| \geq \delta_\ast |S_{\tilde u}(x_0, 1/4 c_1 r^{p_1})|. 
 \label{finalest}
\end{equation}
We have that
\[\frac{(1-\rho)^{-\beta} H_0 G_m
  -\tilde v(x)\tilde G^{Tx_0}}{\left((1+\delta)(1-\rho)^{-\beta} - 1
  \right) H_0} \leq M\]
  is equivalent to 
  \[\tilde v(x) \tilde G^{Tx_0}(x) \geq H_0 \Big[ (1-\rho)^{-\beta} G_m -M[(1+\delta)(1-\rho)^{-\beta} - 1]\Big]=H_0 \Big[ (1-\rho)^{-\beta} G_m-\frac{1}{3}\Big], \]
  so
  \[\tilde v(x) \geq \frac{H_0}{1+\delta} ((1-\rho)^{-\beta} -\frac{1}{3})\geq \frac{H_0}{2}.\]
 Thus, we obtain from (\ref{finalest}) the estimate
\[ |\{\tilde v \geq H_0/2\} \cap S_{\tilde u}(x_0, c_1 r^{p_1})| \geq \delta_\ast |S_{\tilde u}(x_0, c_1 r^{p_1})|.\]
 In view of  \eqref{up_H}, and the volume estimate on sections in (\ref{vol_sect}) (recalling that $\tilde u$ satisfies (\ref{detD2tilde})), we find
$$C t^{-\eps} (2r)^{\beta \eps}\geq  \delta_\ast |S_{\tilde u}(x_0, c_1 r^{p_1})|\geq c_3(n,\lambda,\lambda) r^{np_1/2},$$
for some universally small $c_3$.
By the choice of $\beta$ in (\ref{beta_choice}), and recalling $0<r<1$, we find that $t$ is universally bounded from above by $2^{\beta} (C/c_3)^{\frac{1}{\e}}$. This completes the proof of {\it Step 1}.

{\it Step 2: Proof of the Harnack inequality.}
For each $\tau>0$, let
\[K(\tau) =\inf_{S_{\tilde u}(0, 1/2)} \tilde v  +\tau + 16M ( \|\tilde f\|_{L^n(\tilde S)}+ \sup_{z\in S}\|\tilde f^z\|_{L^n(\tilde S)})/\e_4 \]
and
consider the functions $$\tilde v^{\tau}= \frac{\tilde v}{K(\tau)},\quad 
\text{and }
\tilde f^{\tau}=  \frac{\tilde f}{K(\tau)}.$$
Then, we have the supersolution property
\[\tilde a^{ij} D_{ij}\tilde v^{\tau} + \tilde \bb\cdot D\tilde v^{\tau}\leq| \tilde f^{\tau}|,\]
where
$$\|\tilde b\|_{L^p(\tilde S)} \leq \frac{\e_4}{16M}, \|\tilde f^{\tau}\|_{L^n(\tilde S)}\leq \frac{\e_4}{16M}.$$
For each $z\in S= T(\tilde S)$, consider the functions
\[\tilde\eta^{z,\tau}=  \frac{\tilde \eta^z}{K(\tau) }=\frac{\tilde v \tilde G^{z}}{K(\tau) },\quad
\text{and }
\tilde f^{z,\tau}=  \frac{\tilde f^z}{K(\tau) }.\]
Then, we have the subsolution property
\[\tilde a^{ij} D_{ij}\tilde \eta^{z, \tau} + \tilde \bb^z\cdot D\tilde \eta^{z, \tau}\geq -| \tilde f^{z, \tau}|.\]
Moreover,
$$\sup_{z\in S}\|\tilde \bb^z\|_{L^p(\tilde S)} \leq \frac{\e_4}{16M},\quad \sup_{z\in S}\|\tilde f^{z, \tau}\|_{L^n(\tilde S)}\leq \frac{\e_4}{16M}.$$

We apply the conclusion of {\it Step 1} to $\tilde v^{\tau}$ to obtain
$$\sup_{S_{\tilde u}(0,1/2)} \tilde v \leq C \left(\inf_{S_{\tilde u}(0, 1/2)} \tilde v  +\tau + 16M ( \|\tilde f\|_{L^n(\tilde S)}+ \sup_{z\in S}\|\tilde f^z\|_{L^n(\tilde S)})/\e_4\right).$$
Sending $\tau\rightarrow 0$, we get the Harnack inequality (\ref{HI3}), and completing
the proof of the theorem.
\end{proof}

\subsection{Proof of the twisted Harnack inequality}

\begin{proof}[Proof of Theorem \ref{HI_thm}] 
We prove the theorem when  (\ref{bbzinfty}) is replaced by a slightly relaxed condition (\ref{bbzq}). 
Let $M>1$ be the constant in Proposition \ref{decay_rem}. Let $\e_5(n,p,\lambda,\Lambda,\tl,\tL, R)$ be as in Theorem \ref{bsmall_CG}. Using the H\"older inequality and the volume estimates for sections, we have for $S=S_u(x_\ast, h)$
\[\|\bb\|_{L^{p} (S)} \leq \|\bb\|_{L^{q} (S)} |S|^{\frac{1}{p}-\frac{1}{q}}\leq K|C(n,\lambda,\Lambda) h^{n/2}|^{\frac{1}{p}-\frac{1}{q}}. \]
Similarly,
\[\sup_{z\in S} \|\bb^z\|_{L^{p} (S)} \leq K |C(n,\lambda,\Lambda) h^{n/2}|^{\frac{1}{p}-\frac{1}{q}}.\]
Thus, if $h\leq h_1(n,\lambda,\Lambda,\tl, \tL,  R, q,  K)$, then
\[\|\bb\|_{L^{p} (S)} +\sup_{z\in S} \|\bb^z\|_{L^{p} (S)}  \leq \e_5/2. \]
Since $S_u(\bar x,\bar t)$ is $R$-normalized, by the estimate on the size of sections (see \cite[Lemma 3.52]{LMT}), there exist $\mu(n,\lambda,\Lambda)\in (0, 1)$ and $\bar C=\bar C(n,\lambda,\Lambda, R)$ such that whenever $S_u(x_\ast, h)\subset S_u(\bar x,\bar t/2)$, one has
\[S_u(x_\ast, h)\subset B_{\bar C h^{\mu}} (x_\ast).\]
Now we choose $h_0(n,\lambda,\Lambda,\tl, \tL,  R, q,  K, \gamma)\leq h_1$ so that
\[K (\text{diam}(S_u(x_\ast, h_0)))^\gamma \leq \frac{1}{3M+1}.\]
Then for $h\leq h_0$ and for all $z\in S$, we have
\[\sup_{x\in S}|G^z(x) - G^z(z)| \leq K |x-z|^\gamma\leq K (\text{diam}(S_u(x_\ast, h_0)))^\gamma \leq \frac{1}{3M+1}.\]
The Harnack inequality (\ref{HI_v}) now follows from Theorem \ref{bsmall_CG}.

From the Harnack inequality, we easily obtain the interior H\"older regularity of $v$. The proof is similar to that of \cite[Theorem 2.9]{LMT} for linearized Monge--Amp\`ere equation without drifts, so we omit it.
\end{proof}

{\bf Acknowledgements. } The author would like to thank the referee for carefully reading the paper and providing constructive comments that help improve the exposition of the paper.

\end{document}